\documentclass{amsart}
\usepackage{latexsym}
\usepackage{amssymb}
\usepackage{amsmath}
\usepackage{amsfonts}

\newcommand{\bmat}{\left[\begin{matrix}}
\newcommand{\emat}{\end{matrix}\right]}
\usepackage[latin1]{inputenc} 
\usepackage{amsthm}
\usepackage{amsxtra}
\usepackage{amscd}
\usepackage{setspace}   % enables line spacing commands 
\usepackage{mathrsfs}   % enables \mathscr

\newtheorem{theorem}{Theorem}
\newtheorem{proposition}{Proposition}
\newtheorem{lemma}{Lemma}

\theoremstyle{remark}
\newtheorem*{remark}{Remark}

\theoremstyle{definition}

    \title{On the distribution of lengths of short vectors in a random lattice}
    \author{Seungki Kim}

\begin{document}
\maketitle

\begin{abstract}
We use an idea from sieve theory to estimate the distribution of the lengths of $k$th shortest vectors in a random lattice of covolume 1 in dimension $n$. This is an improvement of the results of Rogers \cite{Rogers} and S\"odergren \cite{Sodergren} in that it allows $k$ to increase with $n$.
\end{abstract}

\section{Introduction}
Let $X_n = SL_n\mathbb{Z} \backslash SL_n\mathbb{R}$ be the space of lattices\footnote{In this paper, a lattice in $\mathbb{R}^n$ is simply a rank $n$ $\mathbb{Z}$-submodule of $\mathbb{R}^n$ (with the standard addition structure).} $L$ of covolume 1 in $\mathbb{R}^n$. $X_n$ admits a unique right $SL_n\mathbb{R}$-invariant probability measure $\mu_n$, derived from a Haar measure of $SL_n\mathbb{R}$ (see \cite{Siegel}). This measure provides the standard notion of a random lattice.

In this paper, we are interested in investigating the statistics of short vectors of a random lattice. Instead of directly stating the mathematical formulation of this problem, we will present a couple of theorems in this direction to give the reader a flavor of this subject. One of the earliest theorems proved concerning lattice statistics is
\begin{theorem}[Siegel \cite{Siegel}] \label{Siegel}
Let $\rho:\mathbb{R}^n \rightarrow \mathbb{R}$ be a compactly supported Borel measurable function. Then
\begin{equation*}
\int_{X_n}\sum_{x \in L\backslash\{0\}} \rho(x) d\mu_n = \int_{\mathbb{R}^n} \rho(x) dx.
\end{equation*}
\end{theorem}
In particular, if $\rho$ is the characteristic function of the ball of radius $r$ centered at the origin, then Siegel's theorem tells us that a random lattice on average has $V(r)$ nonzero vectors of length less than $r$, where $V(r)$ is the volume of a ball with radius $r$.

Later, C.A. Rogers, by using his own generalization of Siegel's theorem above, proved
\begin{theorem}[Rogers \cite{Rogers}] \label{Rogers}
Let $\rho:\mathbb{R}^n \rightarrow \mathbb{R}$ be the characteristic function of a ball of fixed radius $r$ centered at the origin. Fix a positive integer $k$. Provided $n \geq [k^2/4]+3$, we have
\begin{equation*}
e^{-V(r)/2}\sum_{i=0}^\infty \frac{i^k}{i}(V(r)/2)^i \leq \int \left(\frac{1}{2}\sum_{x \in L\backslash\{0\}}\rho(x)\right)^k d\mu(L) \leq e^{-V(r)/2}\sum_{i=0}^\infty \frac{i^k}{i}(V(r)/2)^i + o_n(1).
\end{equation*}
\end{theorem}
This theorem says that the $k$th moment of

\begin{equation*}
\frac{1}{2}\sum_{x \in L\backslash\{0\}}\rho(x) = \mbox{(the number of pairs of vectors $\pm x$ of a lattice $L$ of length $<r$)}
\end{equation*}
converges to the $k$th moment of the Poisson distribution with mean $V(r)/2$ as $n$ goes to infinity. In other words, the number of vectors (identified up to sign) of a random lattice with length less than $r$ has a distribution that converges weakly to the Poisson distribution with mean $V(r)/2$ as dimension becomes large. This result is consistent with the intuition that the first few (relative to the dimension) shortest vectors of a random lattice should be nearly random, as the algebraic structure of a lattice would hardly interfere with the choices of those vectors.

It is clear that one could also convert this data into one about the statistics of the length of $k$th shortest vector (up to sign) of a random lattice, with $k$ fixed and $n$ arbitrarily large. In particular, the case $k=1$, i.e. the statistics of the shortest nonzero vector of a random lattice is very closely related to finding the optimal density of lattice sphere packing.

These and other related theorems were all proved in 1940's and 50's. Since then, the field has come to its mysterious demise, despite much interest in short lattice vectors in computer science and applied mathematics in the latter half of the century. However, in a recent paper, S\"odergren proved that
\begin{theorem}[S\"odergren \cite{Sodergren}] \label{Sodergren}
For a lattice $L$ in $\mathbb{R}^n$ and $t \geq 0$, let $\tilde{N}_t^n(L)$ be the number of nonzero vectors (up to sign) of $L$ in a ball of volume $t$. Taking $L$ to be a random lattice, one may view $\tilde{N}_t^n$ as a stochastic process on the positive real line $\{t \in \mathbb{R}:t\geq 0\}$. As $n \rightarrow \infty$, $\tilde{N}_t^n$ weakly converges to a Poisson process on the positive real line with intensity $1/2$.
\end{theorem}
This result has connections to some topics in analytic number theory, such as zeroes of the Epstein zeta function and to the Berry-Tabor conjecture; for more information see \cite{Sodergren}. S\"odergren also investigates the joint distribution of the angles and the lengths of the first $N$ shortest vectors of a random lattice; see \cite{Sangle}.

The theorems of Rogers and S\"odergren above provide an insight over the ``shape" of a random lattice. Namely, the lengths of the first $k$ shortest vectors of a random lattice of dimension $n$ converge in distribution to the first $k$ points of a Poisson process on the positive real line with intensity $1/2$, with $k$ fixed and as $n$ goes to infinity. Naturally, we would like to remove the condition that $k$ is fixed, and replace it with a stronger condition, such as that $k$ grows with $n$ at a certain rate, in order to understand more fully the statistics of lattice vectors. We expect that as the growth rate of $k$ increases, the Poisson-ness of the length distribution exhibited in the case of fixed $k$ will gradually fade away, as lattices come with the natural algebraic structure, which certainly plays a crucial role in determining their shape. Eventually we hope to grasp this entire picture---the interaction of the inherent structure on lattices (and their moduli space) and their fine quantitative properties---in rigorous terms.

It seems difficult, however, to directly employ Rogers' and S\"odergren's arguments to relax the condition on $k$. Both prove the convergence in distribution by proving the convergence in moments, and the precise quantitative relationship between convergence in moments and convergence in distribution is rather unclear. A more direct proof of their theorems would be helpful. This is the motivation for the present paper.

Using our main theorem, we will be able to obtain the following estimate

\begin{theorem} \label{main_corollary}
Let $S$ be a Borel measurable set in $\mathbb{R}^n$ symmetric at the origin (that is, $x \in S \Leftrightarrow -x \in S$) with Euclidean measure $V$. Suppose that $k \leq (n/2)^{\frac{1}{2} - \varepsilon}$ ($\varepsilon > 0$) is a positive integer, possibly depending on $n$, and suppose also that $8V \leq \sqrt{n/2} - k$. Let $P(S,k)$ be the probability that an $n$-dimensional random lattice has at most $k$ nonzero vectors (up to sign) in $S$. Then $P(S,k)$ is close to $P_{V/2}(k)$ for $n$ sufficiently large, where $P_{V/2}$ is the (left) cumulative distribution function of the Poisson distribution with mean $V/2$. More precisely,

\begin{equation*}
(1-o_n(1))P_{V/2}(k) - o_n(1) \leq P(S,k) \leq (1+o_n(1))P_{V/2}(k) + o_n(1),
\end{equation*}
for all sufficiently large $n$ (depending on $\varepsilon$), where the first $o_n(1)$'s on each side of the inequality can be replaced by $(\sqrt{n/2} - k)^{-1/2}$, and the second $o_n(1)$'s can be replaced by $e(n,k,V)$ (see Theorem \ref{main} below).
\end{theorem}

It is easily seen that Theorem \ref{main_corollary} improves (the implications of) Theorems \ref{Rogers} and \ref{Sodergren} upon letting $S$ be a ball and an annulus, respectively. Its main selling point is that it allows $k$ to increase with $n$ to a certain extent.

We will delegate the proof of Theorem \ref{main_corollary} to the last section of this paper. It is a rather crudely obtained bound from the following theorem, which is the main result of this paper:
\begin{theorem} \label{main}
Let $S$ be as in Theorem \ref{main_corollary}. Let
\begin{equation*}
F_{S,k}(L) = \begin{cases} 1 &\mbox{if $L$ has $\geq k$ nonzero vectors (mod $\pm$) in $S$} \\ 0 &\mbox{otherwise.} \end{cases}
\end{equation*}

For $\alpha, \beta \leq \sqrt{n/2}$ such that $\alpha - k$ is even and $\beta - k$ is odd,
\begin{align*}
\sum_{h=k}^\beta (-1)^{h-k}\frac{(V/2)^h}{h(h-k)!(k-1)!} & - e(n,k,V) \leq \int_{X_n} F_{S,k}(L) d\mu_n \\
& \leq \sum_{h=k}^\alpha (-1)^{h-k}\frac{(V/2)^h}{h(h-k)!(k-1)!} + e(n,k,V),
\end{align*}
where the error term $e(n,k,V)$ has a bound
\begin{equation*}
0 \leq e(n,k,V) \leq \frac{12}{k!}\sqrt{n/2}(0.999)^n(V/2+1)^{\sqrt{n/2}}.
\end{equation*}
Note that, for any $k$ varying between $1$ and $\min(\alpha,\beta)$ (e.g. $k$ could grow with $n$ at a rate comparable to $\sqrt{n}$), and a moderately increasing $V$ (for instance, at the rate of $e^{n^{1/2-\varepsilon}}$ or slower), the right-hand side goes to zero as $n \rightarrow \infty$.
\end{theorem}

The proof of Theorem \ref{main} is in two steps. First we express $F_{S,k}$ as a series using an inclusion-exclusion argument, whose individual terms are integrable over $X_n$ using Rogers' integration formula \cite{Rogersint}. Then we estimate the tail of the series to show that we can ``cut them off'' and obtain an estimate. This is essentially an application of the idea of the Brun sieve. There is a possibility that more elaborate sieve-theoretic ideas will improve the restriction $\alpha, \beta \leq \sqrt{n/2}$ to $\leq n-1$, which will allow even more flexibility on $k$; I hope to return to this topic later.

\section{Rogers' integration formula}

The main technical tool in studying the statistics of lattice vectors is Rogers' integration formula \cite{Rogersint}, which gives an explicit expression for the integrals

\begin{equation} \label{2-1}
\int_{X_n} \sum_{x_1, \ldots, x_k \in L\backslash\{0\}} \rho(x_1, \ldots, x_k)d\mu_n
\end{equation}
\begin{equation} \label{2-2}
\int_{X_n} \sum_{x_1, \ldots, x_k \in L\backslash\{0\} \atop \mathrm{rank}(\langle x_1, \ldots, x_k \rangle)=k} \rho(x_1, \ldots, x_k)d\mu_n
\end{equation}
or the like, where $k < n$ and $\rho$ is a compactly supported Borel-measurable function on $(\mathbb{R}^n)^k$.

\begin{theorem}[Rogers \cite{Rogersint} Section 1, Theorem 4] \label{Rogersint}
(\ref{2-1}) equals
\begin{align*}
& \int_{\mathbb{R}^n} \ldots \int_{\mathbb{R}^n}\rho(x_1, \ldots, x_k)dx_1 \ldots dx_k \\
& + \sum_{(\nu,\mu)}\sum_{q=1}^\infty\sum_D \left(\frac{e_1}{q} \ldots \frac{e_m}{q}\right)^n\int_{\mathbb{R}^n} \ldots \int_{\mathbb{R}^n}\rho\left(\sum_{i=1}^m\frac{d_{i1}}{q}x_i, \ldots, \sum_{i=1}^m\frac{d_{ik}}{q}x_i\right)dx_1\ldots dx_m.
\end{align*}
Here the first sum is over all partitions $(\nu,\mu) = (\nu_1, \ldots, \nu_m; \mu_1, \ldots, \mu_{k-m})$ of the numbers $1 \ldots k$ into two sequences $1 \leq \nu_1 < \ldots < \nu_m \leq k$ and $1 \leq \mu_1 < \ldots < \mu_{k-m} \leq k$ with $1 \leq m \leq k-1$; of course $\nu_i \neq \mu_j$ for any $i,j$. The third sum is taken over all integral $m \times k$ matrices $D$, such that i) no column of $D$ vanishes ii) the greatest common divisor of all entries is 1 iii) for all $i, j$, $D$ satisfies $d_{i\nu_j} = q\delta_{ij}$ and $d_{i\mu_j} = 0$ if $\mu_j < \nu_i$. Finally, $e_i = (\varepsilon_i,q)$, where $\varepsilon_1, \ldots, \varepsilon_m$ are the elementary divisors of $D$.

Furthermore, (\ref{2-2}) equals
\begin{equation*}
\int_{\mathbb{R}^n} \ldots \int_{\mathbb{R}^n}\rho(x_1, \ldots, x_k)dx_1 \ldots dx_k.
\end{equation*}
\end{theorem}

We will apply this theorem to the following situation. For $S \subseteq \mathbb{R}^n$ a Borel measurable set symmetric at the origin, define $S'$ to be the set of elements $x \in S$ whose first nonzero coordinate is positive. In particular, $S'$ does not contain the origin, and every nonzero pair $\{x, -x\}$ in $S$ has only one element in $S'$. Clearly the mass of $S'$ is half of $S$. Let $\chi_{S'}$ be the characteristic function of $S'$, and let
\begin{equation*}
\rho_{S',h}(x_1, \ldots, x_h) = \begin{cases} \prod_{i=1}^h \chi_{S'}(x_i) &\mbox{if $x_i$ are pairwise distinct} \\ 0 &\mbox{otherwise.} \end{cases}
\end{equation*}
We will be interested in estimating
\begin{equation} \label{2-3}
\int_{X_n} \sum_{x_1, \ldots, x_h \in L\backslash\{0\}} \rho_{S',h}(x_1, \ldots, x_h) d\mu_n.
\end{equation}

\begin{proposition} \label{integral}
Let $V$ be the Euclidean volume of $S$. For $n \geq [h^2/4]+3$, the integral (\ref{2-3}) satisfies
\begin{align*}
\left(\frac{V}{2}\right)^h & \leq \int_{X_n} \sum_{x_1, \ldots, x_h \in L\backslash\{0\}} \rho_{S',h}(x_1, \ldots, x_h) d\mu_n \\
& \leq \left(\frac{V}{2}\right)^h + \left(2 \cdot 3^{[h^2/4]}(\sqrt{3}/2)^n + 21 \cdot 5^{[h^2/4]}(1/2)^n\right)\left(\frac{V}{2}+1\right)^h.
\end{align*}
\end{proposition}
\begin{proof}
This is proved by applying Theorem \ref{Rogersint} to $\rho_{S',h}$ and evaluating the following terms separately:
\begin{itemize}
\item The first integral $\int \ldots \int \rho_{S',h}(x_1, \ldots, x_h) dx_1\ldots dx_h$: This is clearly equal to $(V/2)^h$.
\item Summations over $q=1$ and $D$, whose entries are only $0, 1, -1$, and for each column of $D$ exactly one of the entry is nonzero: In this case $x_l = \sum_{i=1}^m\frac{d_{i\nu_l}}{q}x_i = \pm \sum_{i=1}^m\frac{d_{i,\nu_l+1}}{q}x_i = \pm x_l$ for some $l$. If the sign in question is positive, the integral in the summation is zero because the $l$th and $(l+1)$st entries coincide. If it is negative, the integral is still zero because either the $l$th or $(l+1)$st entry is not in $S'$.
\item Summations over $q=1$ and $D$, whose entries are only $0, 1, -1$, and there exists a column of $D$ in which at least two of the entries are nonzero: This is analyzed in \cite{Rogers}, Section 4.
\item Summations over all the rest: This is analyzed in \cite{Rogersmnt}, Section 9. This is where the condition that $n \geq [h^2/4]+3$ is needed; Rogers had to use this assumption in order to show that the summation in question converges. It would be nice to improve this estimate, but I was unable to find a way to do so.
\end{itemize}
\end{proof}

\section{A formula for $F_{S,k}(L)$}

We continue with the notation of the previous sections. For a lattice $L$, define
\begin{align*}
\rho_{S',h}(L) &= (\mbox{the number of subsets of $L \cap S'$ of cardinality $h$}) \\
&= \frac{1}{h!}\sum_{x_1,\ldots,x_h \in L\backslash\{0\}}\rho_{S',h}(x_1,\ldots,x_h).
\end{align*}

Recall that we defined $F_{S,k}(L)$ so that it equals 1 if $L$ has at least $k$ vectors in $S'$, and equals 0 otherwise. The goal of this section is to prove
\begin{proposition} \label{inclexcl}
\begin{equation}
F_{S,k}(L) = \sum_{h=k}^\infty (-1)^{h-k}\binom{h-1}{k-1}\rho_{S',h}(L). \label{F}
\end{equation}
\end{proposition}

\begin{lemma} \label{lemma}
Let $T$ be a finite set. For $R \subseteq S \subseteq T$, define
\begin{equation*}
\mu_S(R) = (-1)^{|S \backslash R|}.
\end{equation*}
Then for any positive integer $k$
\begin{equation*}
\sum_{S \subseteq T \atop |S| \geq k} \sum_{R \subseteq S \atop |R| \geq k} \mu_S(R) = \begin{cases} 1 &\mbox{if $|T| \geq k$} \\ 0 &\mbox{otherwise.} \end{cases}
\end{equation*}
\end{lemma}
\begin{remark}
$\mu_S(R)$ as defined above is the M\"obius function on the lattice (as an order) consisting of the subsets of $T$ ordered by inclusion.
\end{remark}
\begin{proof}
\begin{equation*}
 \sum_{S \subseteq T \atop |S| \geq k} \sum_{R \subseteq S \atop |R| \geq k} \mu_S(R) = \sum_{R \subseteq T \atop |R| \geq k} \sum_{S \subseteq T \atop R \subseteq S} \mu_S(R) = \begin{cases} 1 &\mbox{if $|T| \geq k$} \\ 0 &\mbox{otherwise.} \end{cases}
\end{equation*}
because
\begin{equation*}
\sum_{S \subseteq T \atop R \subseteq S} \mu_S(R) = \begin{cases} 1 &\mbox{if $R = T$} \\ 0 &\mbox{if $R \neq T$} \end{cases}
\end{equation*}
\end{proof}

\begin{proof}[Proof of Proposition \ref{inclexcl}]
Let $T = L \cap S'$.
\begin{align*}
\sum_{S \subseteq T \atop |S| \geq k} \sum_{R \subseteq S \atop |R| \geq k} \mu_S(R) &= \sum_{S \subseteq T \atop |S| \geq k} \sum_{h=k}^{|T|}(-1)^{|S|-h}\binom{|S|}{h} \\
&= \sum_{S \subseteq T \atop |S| \geq k} (-1)^{|S|-k}\binom{|S|-1}{k-1} \\
&= \sum_{h=k}^\infty (-1)^{h-k}\binom{h-1}{k-1}\binom{|T|}{h} \\
&= \sum_{h=k}^\infty (-1)^{h-k}\binom{h-1}{k-1}\rho_{S',h}(L).
\end{align*}
By Lemma \ref{lemma} this completes the proof.
\end{proof}

\section{Estimates}

We are now ready to prove Theorem \ref{main}. Briefly speaking, the strategy is to first show that, for $\alpha \leq \sqrt{n/2}$, the integral of the partial sum of (\ref{F}) over $h \leq \alpha$ converges to the intended main term, and then show that the remaining ``tail'' is either positive or negative depending on the parity of $\alpha - k$. We start by estimating the main term of $\int F_{S,k}d\mu$.

\begin{proposition} \label{head}
Let $k \leq \alpha \leq \sqrt{n/2}$. Then
\begin{equation*}
\left|\int \sum_{h=k}^\alpha (-1)^{h-k}\binom{h-1}{k-1}\rho_{S',h}(L) d\mu - \sum_{h=k}^\alpha \frac{(V/2)^h}{h(h-k)!(k-1)!}\right| \leq e(n,k,V),
\end{equation*}
where the error term $e(n,k,V)$ has a bound
\begin{equation*}
0 \leq e(n,k,V) \leq \frac{12}{k!}\sqrt{n/2}(0.999)^n(V/2+1)^{\sqrt{n/2}}.
\end{equation*}
\end{proposition}
\begin{proof}
By Proposition \ref{integral},
\begin{equation*}
0 \leq \int \rho_{S',h}(L) d\mu - \frac{1}{h!}\left(\frac{V}{2}\right)^h \leq \frac{1}{h!}\left(2 \cdot 3^{[h^2/4]}(\sqrt{3}/2)^n + 21 \cdot 5^{[h^2/4]}(1/2)^n\right)\left(\frac{V}{2}+1\right)^h.
\end{equation*}
For $h \leq \sqrt{n/2}$,
\begin{equation*}
3^{[h^2/4]}(\sqrt{3}/2)^n \leq 3^{n/8}(\sqrt{3}/2)^n = (3^{5/8}/2)^n \leq (0.994)^n
\end{equation*}
and
\begin{equation*}
5^{[h^2/4]}(1/2)^n \leq 5^{n/8}(1/2)^n = (5^{1/8}/2)^n \leq (0.612)^n
\end{equation*}
holds, so
\begin{equation*}
0 \leq \int \rho_{S',h}(L) d\mu - \frac{1}{h!}\left(\frac{V}{2}\right)^h \leq (23/h!)(0.999)^n(V/2+1)^{\sqrt{n/2}}.
\end{equation*}
The proposition now follows easily from this inequality, by summing it up with alternating signs as $h$ runs from $k$ to $\alpha$.
\end{proof}

It remains to estimate the ``tail'':

\begin{proposition} \label{tail}
Let $k \leq \alpha, \beta \leq \sqrt{n/2}$, so that $\alpha-k$ is even and $\beta-k$ is odd. Then
\begin{equation} \label{alpha}
\int \sum_{h=\alpha+1}^\infty (-1)^{h-k}\binom{h-1}{k-1}\rho_{S',h}(L) d\mu \leq 0
\end{equation}
and
\begin{equation} \label{beta}
\int \sum_{h=\beta+1}^\infty (-1)^{h-k}\binom{h-1}{k-1}\rho_{S',h}(L) d\mu \geq 0.
\end{equation}
\end{proposition}
\begin{proof}
Let's prove (\ref{alpha}) first. It suffices to show that for any lattice $L$ with $|L \cap S'| > \alpha$,
\begin{equation*}
\sum_{h=\alpha+1}^\infty (-1)^{h-k}\binom{h-1}{k-1}\rho_{S',h}(L) \leq 0.
\end{equation*}
Write $M = |L \cap S'|$. Then the left-hand side equals
\begin{align*}
& \sum_{h=\alpha+1}^M (-1)^{h-k}\binom{h-1}{k-1}\binom{M}{h} \\
&= \sum_{h=\alpha+1}^M (-1)^{h-k}\frac{h-k+1}{h}\binom{h}{k-1}\binom{M}{h} \\
&= \sum_{h=\alpha+1}^M (-1)^{h-k}\frac{h-k+1}{h}\binom{M}{k-1}\binom{M-k+1}{h-k+1}.
\end{align*}
For convenience, let's denote the summand of the above series by $A_h$.

Case $M \geq 2\alpha -k$: In this case it is clear that $|A_h|$ is increasing for $h = k, \ldots, \alpha.$ Since $\alpha-k$ is even by assumption, $A_\alpha>0$, so $A_\alpha + A_{\alpha-1} > 0$, $A_{\alpha-2} + A_{\alpha-3} > 0$, and so on. Since $A_h$'s are all integers, this implies $\sum_{h=k}^\alpha A_h \geq 1$. Since $\sum_{h=k}^M A_h = 1$, this implies (\ref{alpha}).

Case $\alpha + 1 \leq M \leq 2\alpha-k-1$: In this case we want to show that $A_{\alpha+1} + A_{\alpha+2} < 0$, $A_{\alpha+3} + A_{\alpha+4} < 0$, and so on. This is equivalent to showing
\begin{equation*}
\frac{|A_{h+1}|}{|A_h|} = \frac{h}{h+1}\cdot\frac{h-k+2}{h-k+1}\cdot\frac{M-h}{h-k+2} = \frac{h}{h+1}\cdot\frac{M-h}{h-k+1}<1
\end{equation*}
for $\alpha+1 \leq h \leq M$.
$(M-k)/(h-k+1)$ is the largest when $M$ is the largest and $h$ is the smallest possible, namely when $M = 2\alpha - k - 1$ and $h = \alpha+1$. But even in this case $(M-k)/(h-k+1) = (\alpha - k - 2)/(\alpha - k + 2) < 1$, hence the desired conclusion.

The proof of (\ref{beta}) is more or less the same argument. It suffices to show that for any lattice $L$ with $|L \cap S'| > \beta$,
\begin{equation*}
\sum_{h=\beta+1}^\infty (-1)^{h-k}\binom{h-1}{k-1}\rho_{S',h}(L) \geq 0.
\end{equation*}
By the same argument as earlier we see that this equals
\begin{equation*}
\sum_{h=\beta+1}^M (-1)^{h-k}\frac{h-k+1}{h}\binom{M}{k-1}\binom{M-k+1}{h-k+1}
\end{equation*}
whose summand we again denote by $A_h$.

Case $M \geq 2\beta-k$: Since $|A_h|$ is increasing for $h = k, \ldots, \beta$ and $\beta-k$ is odd, $A_\beta + A_{\beta-1} < 0$, $A_{\beta-2} + A_{\beta-3} < 0$, and so on. By the same logic as earlier (\ref{beta}) follows.

Case $\beta + 1 \leq M \leq 2\beta-k-1$: In this case we want to show that $A_{\beta+1} + A_{\beta+2} > 0$, $A_{\beta+3} + A_{\beta+4} > 0$, and so on. This follows from $|A_{h+1}|/|A_h| < 1$ for $\beta+1 \leq h \leq M$, which we have shown already.
\end{proof}

Theorem \ref{main} now follows trivially from Propositions \ref{head} and \ref{tail}.

\section{A proof of Theorem \ref{main_corollary}}

From Theorem \ref{main} it follows that

\begin{align*}
1 - \sum_{h=k+1}^\alpha \frac{(-1)^{h-k}}{h!}\binom{h-1}{k}(V/2)^h & - e(n,k,V) \leq P(S,k) \\
& \leq 1 - \sum_{h=k+1}^\beta \frac{(-1)^{h-k}}{h!}\binom{h-1}{k}(V/2)^h + e(n,k,V)
\end{align*}
for appropriate $\alpha, \beta$, so it suffices to show that the expression
\begin{equation*}
1 - \sum_{h=k+1}^\alpha \frac{(-1)^{h-k}}{h!}\binom{h-1}{k}(V/2)^h 
\end{equation*}
is close to $P_{V/2}(k)$ given the constraints in the statement of Theorem \ref{main_corollary}. In fact, by introducing the notation
\begin{equation*}
e_\alpha(x) = \sum_{i=0}^\alpha \frac{x^i}{i!},
\end{equation*}
we can write
\begin{equation*}
1 - \sum_{h=k+1}^\alpha \frac{(-1)^{h-k}}{h!}\binom{h-1}{k}(V/2)^h = \sum_{j=0}^k e_{\alpha - j}(-\lambda)\frac{\lambda^j}{j!}.
\end{equation*}

On the other hand, it is a standard fact that
\begin{equation*}
P_\lambda(k) = e^{-\lambda}\sum_{j=0}^k\frac{\lambda^j}{j!}.
\end{equation*}

Therefore it is enough to ensure that $|e^{-\lambda} - e_{\alpha - j}(-\lambda)|$ is small for all $j = 0, \ldots, k$. Writing $m = \alpha - j + 1$, and using Taylor's theorem and Stirling's approximation,

\begin{equation*}
|e^{-\lambda} - e_{\alpha - j}(-\lambda)| \leq \frac{\lambda^{\alpha-j+1}}{(\alpha-j+1)!} \leq \frac{1}{\sqrt{m}}\left(\frac{\lambda e}{m}\right)^m.
\end{equation*}

It can be checked, by taking the log of the above line, that for $m \geq 16\lambda$ (the choice of 16 here is not optimal) we have,
\begin{equation*}
\left(\frac{\lambda e}{m}\right)^m < e^{-\lambda}
\end{equation*}
so that
\begin{equation} \label{bleh}
|e^{-\lambda} - e_{\alpha - j}(-\lambda)| < \frac{1}{\sqrt{m}}e^{-\lambda}.
\end{equation}

Now take $\alpha = \lfloor\sqrt{n/2}\rfloor $ , $\lambda = V/2$. Then whenever $\sqrt{n/2} - k > 8V$, (\ref{bleh}) holds for all $j = 0, \ldots, k$. Choosing $k \leq \sqrt{n/2}^{1-\varepsilon}$ ensures that the right side of (\ref{bleh}) is small compared to $e^{-\lambda}$. This completes the proof of Theorem \ref{main_corollary}.


\begin{thebibliography}{99}

\bibitem{Rogersmnt} C.A. Rogers, The moments of the number of a points of a lattice in a bounded set. Phil. Trans. R. Soc. London. A 248 (1955), 225-251.

\bibitem{Rogersint} C.A. Rogers, Mean values over the space of lattices. Acta Math. 94 (1955), 249-287.

\bibitem{Rogers} C.A. Rogers, The number of lattice points in a set. Proc. Lond. Math. Soc. 6(3) (1956), 249-287.

\bibitem{Siegel} C.L. Siegel, A mean value theorem in geometry of numbers. Ann. of Math. 46(2) (1945), 340-347.

\bibitem{Sodergren} A. S\"odergren, On the Poisson distribution of lengths of lattice vectors in a random lattice. Math. Z. 269 (2011), 945-954.

\bibitem{Sangle} A. S\"odergren, On the distribution of angles between the $N$ shortest vectors in a random lattice. J. London Math. Soc. (2) 84 (2011), 749-764.

\end{thebibliography}
\end{document}